\documentclass[11pt]{amsart}

\usepackage{amsmath}
\usepackage{amssymb}
\usepackage{hyperref}
\usepackage{faktor}
\usepackage{graphicx}
\usepackage{caption}
\usepackage[margin=1.5in]{geometry}

\newtheorem*{nthm*}{Theorem 2.1}

\newtheorem{thm}{Theorem}[section]
\newtheorem{lem}[thm]{Lemma}

\newtheorem{cor}[thm]{Corollary}

\title[Probability of $k$-wise Relatively $r$-prime Ideals]{The Probability that Ideals in a Number Ring are $k$-wise Relatively $r$-prime}
\author{Brian D. Sittinger and Ryan D. DeMoss}
\begin{document}

\vspace*{-1.5cm}

\begin{abstract}

\noindent We say that $n$ nonzero ideals of algebraic integers in a fixed number ring are $k$-wise relatively $r$-prime if any $k$ of them are relatively $r$-prime. In this article, we provide an exact formula for the probability that $n$ nonzero ideals of algebraic integers in a fixed number ring are $k$-wise relatively $r$-prime.

\end{abstract}

\maketitle

\normalsize

\section{Introduction}  \label{sect1}

A rather intriguing result in number theory is that the probability that two random positive integers are relatively prime is ${6}/{\pi^2}$. More generally, the probability that $n$ positive integers are relatively prime is ${1}/{\zeta(n)}$; see \cite{Nymann} for instance. In 1975, Benkoski \cite{Benkoski} further generalized Lehmer's work, showing that the probability that $n$ positive integers are relatively $r$-prime (that is, these integers have no common $r$th power prime factor) is ${1}/{\zeta(rn)}$.

Another way to generalize this is to consider pairwise relatively prime integers. In 2002, L\'aszl\'o T\'oth \cite{Toth1} established via recursion that the probability that $n$ positive integers are \textit{pairwise} relatively prime equals
$$\prod_{p \;\text{prime}} \Big(1 - \frac{1}{p}\Big)^{n-1} \Big(1 + \frac{n-1}{p}\Big).$$
In 2012, Jerry Hu \cite{Hu} generalized T\'oth's work by finding a formula for the probability that $n$ positive integers are $k$-wise relatively prime, concluding that this probability equals
$$\prod_{p \;\text{prime}} \sum_{j=0}^{k-1} \binom{n}{j} \Big(1 - \frac{1}{p}\Big)^{n-j} \frac{1}{p^j}.$$
\noindent Subsequently in \cite{Toth2}, not only did T\'oth give a non-recursive proof to Hu's fact, but he also improved the error term that estimated this probability when the positive integers are at most $x$ before letting $x \to \infty$. %Give the error term version?

Observing that these results all refer to the set of positive integers, it is natural to ask if one can extend these results to other sets with similar properties. In 2010, Sittinger \cite{Sittinger} extended the result of Benkoski to ideals of algebraic integers in a number field using Nymann's methods of deriving growth estimates. He deduced the probability that $n$ ideals of algebraic integers from a given algebraic number ring $\mathcal{O}_K$ are relatively $r$-prime is $1/\zeta_K(rn)$, where $\zeta_K$ denotes the Dedekind zeta function.  In this article, we adapt the method in \cite{Toth2} to provide a generalization of the results of T\'oth and Hu to rings of algebraic integers.

\section{Statement of the Main Theorem}

Let $K$ be an algebraic number field with associated ring of algebraic integers $\mathcal{O}$. Fixing a positive integer $r$, we say that the nonzero ideals $\mathfrak{a}_1, \dots, \mathfrak{a}_n \subseteq \mathcal{O}$ are \textbf{relatively $r$-prime} if there does not exist a prime ideal $\mathfrak{p} \subseteq \mathcal{O}$ such that $\mathfrak{p}^r \mid \mathfrak{a}_i$ for each $i = 1, 2, \dots, n$. Moreover, for any fixed positive integer $k \leq n$, we say that $\mathfrak{a}_1, \dots, \mathfrak{a}_n$ are \textbf{$k$-wise relatively $r$-prime} if any $k$ of these ideals are relatively $r$-prime. Observe that when $r = 1$, these definitions reduce to nonzero ideals being relatively prime and $k$-wise relatively prime, respectively.

We next state the the main result of this article (which is proved in the following section). To this end, given $n$ nonzero ideals $\mathfrak{a}_1, \dots, \mathfrak{a}_n \subseteq \mathcal{O}$, define the characteristic function
$$\rho_{n,k,r}(\mathfrak{a}_1, \dots, \mathfrak{a}_n) = \begin{cases} 1 & \text{if } \mathfrak{a}_1, \dots, \mathfrak{a}_n \text{ are $k$-wise relatively $r$-prime,}\\ 0 & \text{otherwise.}\end{cases}$$
\noindent Moreover, we use $\mathfrak{N}(\mathfrak{a}) = |\mathcal{O}/\mathfrak{a}|$ to denote the norm of an ideal $\mathfrak{a} \subseteq \mathcal{O}$.

\begin{thm} \label{thm:main}
Fix positive integers $n,k,r$ where $n \geq k \geq 2$, and let $K$ be an algebraic number field of degree $d$ over $\mathbb{Q}$ whose ring of integers is $\mathcal{O}$. Let $\epsilon =  \displaystyle\frac{2}{d+1}$. Then, there exists a constant $c > 0$ only dependent on $K$ such that
$$\sum_{\mathfrak{N}(\mathfrak{a}_1), \dots, \mathfrak{N}(\mathfrak{a}_n) \leq x} \rho_{n,k,r}(\mathfrak{a}_1, \dots, \mathfrak{a}_n) = P_{n,k,r} \cdot (cx)^n + O(R_{n,k,r}(x)),$$
\noindent where $$P_{n,k,r} = \prod_{\mathfrak{p}} \sum_{j=0}^{k-1} \binom{n}{j}
\Big(1 - \frac{1}{\mathfrak{N}(\mathfrak{p}^r)}\Big)^{n-j} \frac{1}{\mathfrak{N}(\mathfrak{p}^r)^j},$$
\noindent and $$R_{n,k,r}(x) = \begin{cases} x^{n-\epsilon} \log^{n-1}{x} & \text{if } k = 2 \text{ and } r = 1,
\\ x^{n-\epsilon} & \text{otherwise}.\end{cases}$$
\end{thm}

\noindent \textbf{Remark:} The values of $c$ and $\epsilon$ in Theorem 2.1 come from the following classic result of Landau \cite{Landau} that estimates $H(x)$, the number of ideals in $\mathcal{O}$ with norm at most $x$.

\begin{lem} \label{thm:Ideals}
Let $[K:\mathbb{Q}] = d = s + 2t$, where $s$ and $t$ respectively denote the number of real and pairs of complex embeddings of $K$, and let $g$ denote the number of roots of unity in $\mathcal{O}$. Then, we have
$$H(x) = cx + O(x^{1-\epsilon}),$$
where $c = \displaystyle\frac{2^{s+t} \pi^t hR}{g \sqrt{|\Delta|}}$, and $\epsilon =  \displaystyle\frac{2}{d+1}$.
\end{lem}

\noindent As usual, $h$, $R$, and $\Delta$ refer to the class number, regulator, and discriminant of the number field $K$, respectively.

As a consequence of this theorem, we readily have the following probability formula that is the desired number ring generalization of the results of T\'oth and Hu that were proved exclusively in the context of the set of integers.

\begin{cor}
Fix positive integers $n,k,r$ where $n \geq k \geq 2$. Then, the probability that $n$ nonzero ideals of algebraic integers in $\mathcal{O}$ are $k$-wise relatively $r$-prime equals
$$P_{n,k,r} = \prod_{\mathfrak{p}} \sum_{j=0}^{k-1} \binom{n}{j}
\Big(1 - \frac{1}{\mathfrak{N}(\mathfrak{p}^r)}\Big)^{n-j} \frac{1}{\mathfrak{N}(\mathfrak{p}^r)^j}.$$
\end{cor}

\begin{proof}
By Theorem \ref{thm:main} and Lemma \ref{thm:Ideals}, the desired probability equals
$$\lim_{x \to \infty} \frac{\sum_{\mathfrak{N}(\mathfrak{a}_1), \dots, \mathfrak{N}(\mathfrak{a}_n) \leq x} \rho_{n,k,r}(\mathfrak{a}_1, \dots, \mathfrak{a}_n)}{H(x)^n} = \lim_{x \to \infty} \frac{P_{n,k,r} \cdot (cx)^n + O(R_{n,k,r}(x))}{(cx + O(x^{1 - \epsilon}))^n} = P_{n,k,r}.$$
\end{proof}

\section{Proof of the Main Theorem}

Let $n$, $k$, and $r$ be positive integers such that $n \geq k \geq 2$. Given $n$ nonzero ideals $\mathfrak{a}_1, \dots, \mathfrak{a}_n \subseteq \mathcal{O}$, we consider the characteristic function
$$\rho_{n,k,r}(\mathfrak{a}_1, \dots, \mathfrak{a}_n) = \begin{cases} 1 & \text{if } \mathfrak{a}_1, \dots, \mathfrak{a}_n \text{ are $k$-wise relatively $r$-prime,}\\ 0 & \text{otherwise.}\end{cases}$$

\noindent Observe that $\rho_{n,k,r}$ is multiplicative: If $\gcd(\mathfrak{a}_1 \dots \mathfrak{a}_n, \mathfrak{b}_1 \dots \mathfrak{b}_n) = (1)$, then
$$\rho_{n,k,r}(\mathfrak{a}_1 \mathfrak{b}_1, \dots, \mathfrak{a}_n \mathfrak{b}_n) = \rho_{n,k,r}(\mathfrak{a}_1, \dots, \mathfrak{a}_n) \rho_{n,k,r}(\mathfrak{b}_1, \dots, \mathfrak{b}_n).$$

\noindent Using the notation $\mathfrak{a} = \prod_{\mathfrak{p}} \mathfrak{p}^{\nu_{\mathfrak{p}}(\mathfrak{a})}$ for the prime ideal factorization of the nonzero ideal $\mathfrak{a} \subseteq \mathcal{O}$ (where all but finitely many $\nu_{\mathfrak{p}}(\mathfrak{a})$ equal 0), we deduce that for every nonzero ideals $\mathfrak{a}_1, \dots \mathfrak{a}_n \subseteq \mathcal{O}$:
$$\rho_{n,k,r}(\mathfrak{a}_1, \dots, \mathfrak{a}_n) = \prod_{\mathfrak{p}} \rho_{n,k,r}(\mathfrak{p}^{\nu_{\mathfrak{p}}(\mathfrak{a}_1)}, \dots, \mathfrak{p}^{\nu_{\mathfrak{p}}(\mathfrak{a}_n)}).$$

\noindent Moreover, for every $\nu_1, \dots, \nu_n \geq 0$,
$$\rho_{n,k,r}(\mathfrak{p}^{\nu_1}, \dots, \mathfrak{p}^{\nu_n}) = \begin{cases}
1 & \text{if there are at most $k-1$ values $\nu_i \geq r$},\\ 0 & \text{otherwise.} \end{cases}$$

\noindent We first state a lemma that gives a convergence result that we repeatedly use. Let
$$e_j(x_1, \dots, x_n) = \sum_{1 \leq i_1 < \dots < i_j \leq n} x_{i_1} \dots x_{i_j}$$
\noindent denote the elementary symmetric polynomial in $x_1, \dots, x_n$ of total degree $j$.

\begin{lem}\label{thm:convergence}
Fix positive integers $n,k,r$ where $n \geq k \geq 2$, and let $s_1, \dots, s_n \in \mathbb{C}$. If $\emph{Re}(s_1), \dots,
\emph{Re}(s_n) > 1$, then
$$\sum_{\mathfrak{a}_1, \dots, \mathfrak{a}_n} \frac{\rho_{n,k,r}(\mathfrak{a}_1, \dots, \mathfrak{a}_n)}{\mathfrak{N}(\mathfrak{a}_1)^{s_1} \dots \mathfrak{N}(\mathfrak{a}_n)^{s_n}}
= \zeta_K (s_1) \dots \zeta_K (s_n) D_{n,k,r}(s_1, \dots, s_n),$$
\noindent where
$$D_{n,k,r}(s_1, \dots, s_n) = \prod_{\mathfrak{p}} \Big[1 - \sum_{j=k}^n (-1)^{j-k} \binom{j-1}{k-1} e_j\Big(\frac{1}{\mathfrak{N}(\mathfrak{p})^{rs_1}}, \dots, \frac{1}{\mathfrak{N}(\mathfrak{p})^{rs_n}}\Big)\Big],$$

\noindent and $D_{n,k,r}(s_1, \dots, s_n)$ is absolutely convergent if and only if $\emph{Re}(s_{i_1} + \dots + s_{i_j}) > \frac{1}{r}$ for every $1 \leq i_1 < \dots < i_j \leq n$ with $k \leq j \leq n$.
\end{lem}

\begin{proof}
Since $\rho_{n,k,r}$ is a multiplicative function, we expand the given Dirichlet series into an infinite product:
$$\sum_{\mathfrak{a}_1, \dots, \mathfrak{a}_n} \frac{\rho_{n,k,r}(\mathfrak{a}_1, \dots, \mathfrak{a}_n)}{\mathfrak{N}(\mathfrak{a}_1)^{s_1} \dots \mathfrak{N}(\mathfrak{a}_n)^{s_n}}
= \prod_{\mathfrak{p}} \sum_{\nu_1, \dots, \nu_n} \frac{\rho_{n,k,r}(\mathfrak{p}^{\nu_1}, \dots, \mathfrak{p}^{\nu_n})}{\mathfrak{N}(\mathfrak{p})^{\nu_1 s_1 + \dots + \nu_n s_n}}.$$

\noindent To expand this further, we let $T$ denote a subset of $\{1, 2, \dots, n\}$. With this convention, the last expression can be rewritten as
\begin{align*} & \prod_{\mathfrak{p}} \Big[\sum_{j = 1}^{k-1} \sum_{|T| = j} \sum_{\substack{v_i \geq r \text{ if } i \in T\\ 0 \leq v_i \leq r-1 \text{ if } i \notin T}} \frac{1}{\mathfrak{N}(\mathfrak{p})^{\nu_1 s_1 + \dots + \nu_n s_n}}\Big]\\
&= \prod_{\mathfrak{p}} \Big[\sum_{j = 1}^{k-1} \sum_{|T| = j} \prod_{i \in T} \frac{1}{\mathfrak{N}(\mathfrak{p})^{rs_i}} \Big(1 - \frac{1}{\mathfrak{N}(\mathfrak{p})^{s_i}}\Big)^{-1} \cdot \prod_{i \notin T} \Big(1 - \frac{1}{\mathfrak{N}(\mathfrak{p})^{rs_i}} \Big) \Big(1 - \frac{1}{\mathfrak{N}(\mathfrak{p})^{s_i}}\Big)^{-1}\Big]\\
&=  \zeta_K (s_1) \dots \zeta_K (s_n) \prod_{\mathfrak{p}} \Big[\sum_{j = 1}^{k-1} \sum_{|T| = j} \prod_{i \in T} \frac{1}{\mathfrak{N}(\mathfrak{p})^{rs_i}} \cdot \prod_{i \notin T} \Big(1 - \frac{1}{\mathfrak{N}(\mathfrak{p})^{rs_i}} \Big)\Big]\\
&= \zeta_K (s_1) \dots \zeta_K (s_n) D_{n,k,r}(s_1, \dots, s_n).
\end{align*}

\noindent The last line directly follows from Lemma 3.2 in T\'oth \cite{Toth2}.
The absolute convergence of the given multivariate Dirichlet series as well as that of $D_{n,k,r}$ come directly from standard facts concerning absolute convergence of univariate Dirichlet series.
\end{proof}

Next, we define
$$\psi_{n,k,r}(\mathfrak{a}_1, \dots, \mathfrak{a}_n) = \sum_{\substack{\mathfrak{a}_i = \mathfrak{d}_i \mathfrak{e}_i \\ i = 1, \dots, n}} \rho_{n,k,r}(\mathfrak{d}_1, \dots, \mathfrak{d}_n) \mu(\mathfrak{e}_1) \dots \mu(\mathfrak{e}_n).$$
By using M\"{o}bius inversion, this is equivalent to writing
$$\rho_{n,k,r}(\mathfrak{a}_1, \dots, \mathfrak{a}_n) = \sum_{\substack{\mathfrak{d}_i \mid \mathfrak{a}_i\\ i = 1, \dots, n}} \psi_{n,k,r}(\mathfrak{d}_1, \dots, \mathfrak{d}_n).$$

\noindent It is clear that $\psi_{n,k,r}$ is symmetric in its arguments. Moreover, since $\rho_{n,k,r}$ is multiplicative, it follows that $\psi_{n,k,r}$ is also multiplicative. Hence, it suffices to find the value of $\rho_{n,k,r}$ when its arguments are all powers of the same prime ideal $\mathfrak{p}$.

\begin{lem}\label{thm:psi}
If $\mathfrak{p}$ be a prime ideal in a number ring $\mathcal{O}$ and $\nu_1, \dots, \nu_n$ are non-negative integers, then
$$\psi_{n,k,r}(\mathfrak{p}^{\nu_1}, \dots, \mathfrak{p}^{\nu_n}) = \begin{cases}
1 & \text{if all } \nu_i = 0,\\
(-1)^{j-k+1} \binom{j-1}{k-1} & \text{if all } \nu_i  \in \{0, r\}, \text{ and } j = \sum_{i = 1}^n \nu_i \geq rk,\\
0 & \text{otherwise.}
\end{cases}$$
\end{lem}

\begin{proof}
In order to prove this lemma, we first observe that
$$\psi_{n,k,r}(\mathfrak{p}^{\nu_1}, \dots, \mathfrak{p}^{\nu_n}) = \sum_{\substack{0 \leq j_i \leq \nu_i\\ i = 1, \dots, n}} \rho_{n,k,r}(\mathfrak{p}^{j_1}, \dots, \mathfrak{p}^{j_n}) \mu(\mathfrak{p}^{\nu_1 - j_1}) \dots \mu(\mathfrak{p}^{\nu_n - j_n}).$$

First, suppose that one of $\nu_i \geq r+1$; without loss of generality, assume that $\nu_1 \geq r+1$. Then, $\psi_{n,k,r}(\mathfrak{p}^{\nu_1}, \dots, \mathfrak{p}^{\nu_n})$ equals

\begin{align*}
&\sum_{\substack{0 \leq j_i \leq \nu_i\\ i = 2, \dots, n}} \Big[\rho_{n,k,r}(\mathfrak{p}^{v_1 - 1}, \mathfrak{p}^{j_2}, \dots, \mathfrak{p}^{j_n}) \mu(\mathfrak{p}) + \rho_{n,k,r}(\mathfrak{p}^{v_1}, \mathfrak{p}^{j_2}, \dots, \mathfrak{p}^{j_n}) \mu((1))\Big] \mu(\mathfrak{p}^{\nu_2 - j_2}) \dots \mu(\mathfrak{p}^{\nu_n - j_n})\\
&= \sum_{\substack{0 \leq j_i \leq \nu_i\\ i = 2, \dots, n}} \rho_{n,k,r}(\mathfrak{p}^{v_1}, \mathfrak{p}^{j_2}, \dots, \mathfrak{p}^{j_n}) \Big[\mu(\mathfrak{p}) + \mu((1))\Big]  \mu(\mathfrak{p}^{\nu_2 - j_2}) \dots \mu(\mathfrak{p}^{\nu_n - j_n}) \\ &= 0.\end{align*}

\noindent Similarly, if one of $\nu_i \in \{1, 2, \dots, r-1\}$, then $\psi_{n,k,r}(\mathfrak{p}^{\nu_1}, \dots, \mathfrak{p}^{\nu_n}) = 0$.

Hereafter, we assume that each $\nu_i \in \{0, r\}$. Since the claim is clearly true if all $\nu_i = 0$, assume that there are $j \geq 1$ arguments of $\psi_{n,k,r}$ for which $\nu_i = r$; without loss of generality we consider $\psi_{n,k,r}(\underbrace{\mathfrak{p}^r, \dots , \mathfrak{p}^r}_{j \text{ times}}, \underbrace{(1), \dots, (1)}_{n-j \text{ times}})$. Then, we have
$$\psi_{n,k,r}(\underbrace{\mathfrak{p}^r, \dots , \mathfrak{p}^r}_{j}, \underbrace{(1), \dots, (1)}_{n-j}) = \sum_{i = 0}^j \binom{j}{i} \rho_{n,k,r}(\underbrace{\mathfrak{p}^r, \dots , \mathfrak{p}^r}_{i}, \underbrace{(1), \dots, (1)}_{n-i}) \mu(\mathfrak{p})^{j-i}.$$

\noindent If $j < k$, then the Binomial Theorem yields
$$\psi_{n,k,r}(\underbrace{\mathfrak{p}^r, \dots , \mathfrak{p}^r}_{j}, \underbrace{(1), \dots, (1)}_{n-j}) = \sum_{i = 0}^j (-1)^{j-i}  \binom{j}{i} = 0.$$

\noindent If $j \geq k$, then
$$\psi_{n,k,r}(\underbrace{\mathfrak{p}^r, \dots , \mathfrak{p}^r}_{j}, \underbrace{(1), \dots, (1)}_{n-j}) =
\Big[\sum_{i = 0}^{k-1} (-1)^{j-i}  \binom{j}{i}\Big] + 0 = (-1)^{j-k+1} \binom{j-1}{k-1}.$$

\noindent Note that for the last equality, we used the identity (where $0 \leq d < n$)
$$\sum_{i = 0}^d (-1)^i \binom{n}{i} = (-1)^d \binom{n-1}{d}.$$
\end{proof}

Before we finally prove the main theorem of this article, we need to establish the following combinatorial identity.

\begin{lem}\label{thm:binomial}
Let $n$ and $k$ be positive integers where $n \geq k$.  Then,
$$\sum_{j=k}^n (-1)^{j-k} \binom{n}{j} \binom{j-1}{k-1} x^j = \sum_{j=k}^n \binom{n}{j} x^j (1-x)^{n-j}.$$
\end{lem}

\begin{proof}
By expanding $(1-x)^{n-j}$ with the Binomial Theorem and collecting like terms, we obtain
\begin{align*}
\sum_{j=k}^n \binom{n}{j} x^j (1-x)^{n-j}
&= \sum_{j=k}^n \Big[\sum_{l=0}^{j-k} (-1)^{j-k+l} \binom{n}{k+l} \binom{n-k-l}{r-j}\Big] x^j\\
&= \sum_{j=k}^n \Big[\sum_{l=0}^{j-k} (-1)^{j-k+l} \binom{n}{j} \binom{j}{k+l}\Big] x^j.
\end{align*}

\noindent Therefore, it suffices to prove that for all $j = k , \dots, n$:
$$\binom{j-1}{k-1} = \sum_{l=0}^{j-k} (-1)^l \binom{j}{k+l}.$$

\noindent This latter identity immediately follows rewriting the series as a telescoping sum via Pascal's identity $$\binom{j}{k+l} = \binom{j-1}{k+l} + \binom{j-1}{k+l-1}.$$
\end{proof}

%\noindent We are now ready to prove the main theorem of this section. For convenience, we restate it %here before proving it.

%\begin{nthm*}\label{thm:enumeration}
%Fix positive integers $n,k,r$ where $n \geq k \geq 2$. Then, we have
%$$\sum_{\mathfrak{N}(\mathfrak{a}_1), \dots, \mathfrak{N}(\mathfrak{a}_n) \leq x} %\rho_{n,k,r}(\mathfrak{a}_1, \dots, \mathfrak{a}_n) = P_{n,k,r} \cdot (cx)^n + O(R_{n,k,r}(x)),$$
%\noindent where $$P_{n,k,r} = \prod_{\mathfrak{p}} \sum_{j=0}^{k-1} \binom{n}{j}
%\Big(1 - \frac{1}{\mathfrak{N}(\mathfrak{p}^r)}\Big)^{n-j} \frac{1}{\mathfrak{N}(\mathfrak{p}^r)^j},$$
%\noindent and $$R_{n,k,r}(x) = \begin{cases} x^{n-\epsilon} \log^{n-1}{x} & \text{if } k = 2 \text{ and %} r = 1,
%\\ x^{n-\epsilon} & \text{otherwise}.\end{cases}$$
%\end{nthm*}

\begin{proof} [Proof of Theorem \ref{thm:main}]

%Summarize attack of the proof!

First of all, we apply Lemma \ref{thm:psi} to obtain

%\resizebox{0.96\linewidth}{!}{\begin{minipage}{\linewidth}
$$\sum_{\mathfrak{a}_1, \dots, \mathfrak{a}_n} \frac{\psi_{n,k,r}(\mathfrak{a}_1, \dots, \mathfrak{a}_n)}{\mathfrak{N}(\mathfrak{a}_1)^{s_1} \dots \mathfrak{N}(\mathfrak{a}_n)^{s_n}}
= \prod_{\mathfrak{p}} \sum_{\nu_1, \dots, \nu_n=0}^{\infty} \frac{\psi_{n,k,r}(\mathfrak{p}^{\nu_1}, \dots, \mathfrak{p}^{\nu_n})}{\mathfrak{N}(\mathfrak{p})^{r\nu_1 s_1 + \dots + r\nu_n s_n}} = D_{n,k,r}(s_1, \dots, s_n).$$

%\end{minipage}}

\noindent Then by using Lemma \ref{thm:Ideals}, we find that $\displaystyle \sum_{\substack{\mathfrak{N}(\mathfrak{a}_i) \leq x\\i = 1, \dots, n}} \rho_{n,k,r}(\mathfrak{a}_1, \dots, \mathfrak{a}_n)$ equals
\begin{align*}
&\sum_{\substack{\mathfrak{N}(\mathfrak{d}_i) \leq x\\i = 1, \dots, n}} \psi_{n,k,r}(\mathfrak{d}_1, \dots, \mathfrak{d}_n) H\Big(\frac{x}{\mathfrak{N}(\mathfrak{d}_1)}\Big) \dots H\Big(\frac{x}{\mathfrak{N}(\mathfrak{d}_n)}\Big)\\
&= \sum_{\substack{\mathfrak{N}(\mathfrak{d}_i) \leq x\\i = 1, \dots, n}} \psi_{n,k,r}(\mathfrak{d}_1, \dots, \mathfrak{d}_n) \prod_{j = 1}^n \Big[\frac{cx}{\mathfrak{N}(\mathfrak{d}_j)} + O\Big(\Big(\frac{x}{\mathfrak{N}(\mathfrak{d}_j)} \Big)^{1-\epsilon}\Big)\Big]\\
&= (cx)^n \sum_{\substack{\mathfrak{N}(\mathfrak{d}_i) \leq x\\i = 1, \dots, n}}
\frac{\psi_{n,k,r}(\mathfrak{d}_1, \dots, \mathfrak{d}_n)}{\mathfrak{N}(\mathfrak{d}_1) \dots \mathfrak{N}(\mathfrak{d}_n)} + Q_{n,k,r}(x),
\end{align*}

\noindent where
$$Q_{n,k,r}(x) \ll \sum_{\substack{u_1, \dots, u_n \in \{0,1\}\\ \text{at least one } u_i = 0}} x^{u_1 + \dots + u_n} \cdot x^{1-\epsilon} \sum_{\substack{\mathfrak{N}(\mathfrak{d}_i) \leq x\\i = 1, \dots, n}}
\frac{|\psi_{n,k,r}(\mathfrak{d}_1, \dots, \mathfrak{d}_n)|}{\mathfrak{N}(\mathfrak{d}_1)^{u_1} \dots \mathfrak{N}(\mathfrak{d}_n)^{u_n}}.$$

First, we estimate $Q_{n,k,r}(x)$. Without loss of generality, fix $u_1, \dots, u_n \in \{0, 1\}$, where $u_n = 0$. Then, letting
$$A = x^{u_1 + \dots + u_n} \cdot \, x^{1-\epsilon} \sum_{\substack{\mathfrak{N}(\mathfrak{d}_i) \leq x\\i = 1, \dots, n}} \frac{|\psi_{n,k,r}(\mathfrak{d}_1, \dots, \mathfrak{d}_n)|}{\mathfrak{N}(\mathfrak{d}_1)^{u_1} \dots \mathfrak{N}(\mathfrak{d}_n)^{u_n}},$$
\noindent we obtain
$$A \leq x^{n-\epsilon} \sum_{\substack{\mathfrak{N}(\mathfrak{d}_i) \leq x\\i = 1, \dots, n}}
\frac{|\psi_{n,k,r}(\mathfrak{d}_1, \dots, \mathfrak{d}_n)|}{\mathfrak{N}(\mathfrak{d}_1) \dots \mathfrak{N}(\mathfrak{d}_{n-1})}.$$

First, we suppose that $k \geq 3$. Then since $D_{n,k,r}(1, \dots, 1, 0)$ is absolutely convergent for $k \geq 3$ by Lemma \ref{thm:convergence}, we find that $A = O(x^{n - \epsilon}).$ Hence, we conclude that $Q_{n,k,r}(x) = O(x^{n - \epsilon})$.

Next suppose that $k = 2$. Then, we have
\begin{align*}
A &\leq x^{n - \epsilon} \prod_{\mathfrak{N}(\mathfrak{p}) \leq x} \sum_{\nu_1, \dots, \nu_n = 0}^{\infty}
\frac{|\psi_{n,2,r}(\mathfrak{p}^{\nu_1}, \dots, \mathfrak{p}^{\nu_n})|}{\mathfrak{N}(\mathfrak{p})^{\nu_1 + \dots + \nu_{n-1}}}\\
&= x^{n - \epsilon} \prod_{\mathfrak{N}(\mathfrak{p}) \leq x} \Big(1 + \frac{n-1}{\mathfrak{N}(\mathfrak{p}^r)}
+ \frac{c_2}{\mathfrak{N}(\mathfrak{p}^r)^2} + \dots + \frac{c_{n-1}}{\mathfrak{N}(\mathfrak{p}^r)^{n-1}} \Big)
\end{align*}

\noindent for some positive integers $c_2, \dots, c_{n-1}$. Note that we have exactly $\mathfrak{N}(\mathfrak{p}^r)$ in the denominator if and only if $\nu_n = r$ and exactly one of $\nu_1, \dots, \nu_{n-1}$ equals $r$ with the rest of them equaling 0; this occurs $n-1$ times. Therefore,
$$A = O\Big(x^{n - \epsilon} \prod_{\mathfrak{N}(\mathfrak{p}) \leq x} \Big(1 + \frac{1}{\mathfrak{N}(\mathfrak{p}^r)}\Big)^{n-1}\Big).$$
%Fix!
\noindent If $r \geq 2$, this latter product is $O(1)$, and thus $A = O(x^{n - \epsilon})$. If $r = 1$, we use Mertens' theorem for number fields \cite{Rosen}
$$\prod_{\mathfrak{N}(\mathfrak{p}) \leq x} \Big(1 - \frac{1}{\mathfrak{N}(\mathfrak{p})}\Big)^{-1} = e^{\gamma} \alpha_K \log{x} + O(1),$$

\noindent where $\alpha_K$ is the residue of $\zeta_{K}(s)$ at $s = 1$, and we find that $A = O(x^{n - \epsilon} \log^{n-1}{x})$. Hence, we conclude that
$$Q_{n,2,r}(x) = \begin{cases} O(x^{n - \epsilon} \log^{n-1}{x}) & \text{if } r = 1,\\ O(x^{n - \epsilon}) & \text{if } r \geq 2. \end{cases}$$

In summary, we have found that
$$Q_{n,k,r}(x) = \begin{cases}  O(x^{n - \epsilon} \log^{n-1}{x}) & \text{if } k = 2 \text{ and } r = 1,\\
O(x^{n - \epsilon}) & \text{otherwise.}\\
\end{cases}$$

\noindent Next, we turn our attention to examining the sum
$$\sum_{\substack{\mathfrak{N}(\mathfrak{d}_i) \leq x\\i = 1, \dots, n}}
\frac{\psi_{n,k,r}(\mathfrak{d}_1, \dots, \mathfrak{d}_n)}{\mathfrak{N}(\mathfrak{d}_1) \dots \mathfrak{N}(\mathfrak{d}_n)}.$$

We first rewrite this as
$$\sum_{\mathfrak{d}_1, \dots, \mathfrak{d}_n}
\frac{\psi_{n,k,r}(\mathfrak{d}_1, \dots, \mathfrak{d}_n)}{\mathfrak{N}(\mathfrak{d}_1) \dots \mathfrak{N}(\mathfrak{d}_n)} -
\sum_{\substack{I \subseteq \{1, \dots, n\} \\I \neq \varnothing}}
\sum_{\substack{\mathfrak{N}(\mathfrak{d}_i) > x, \, i \in I \\ \mathfrak{N}(\mathfrak{d}_j) \leq x, \, j \notin I}}
\frac{\psi_{n,k,r}(\mathfrak{d}_1, \dots, \mathfrak{d}_n)}{\mathfrak{N}(\mathfrak{d}_1) \dots \mathfrak{N}(\mathfrak{d}_n)}.$$

\noindent The first series is convergent by Lemma \ref{thm:convergence} and equals
$$D_{n,k,r}(1, \dots, 1) = \prod_{\mathfrak{p}} \Big(1 - \sum_{j=k}^n (-1)^{j-k} \binom{n}{j} \binom{j-1}{k-1} \frac{1}{\mathfrak{N}(\mathfrak{p}^r)^j} \Big).$$
\noindent Moreover, this product equals $P_{n,k,r}$, and we can see this from rewriting $P_{n,k,r}$ as
$$P_{n,k,r} = \prod_{\mathfrak{p}} \Big[1 - \sum_{j=k}^n \binom{n}{j}
\Big(1 - \frac{1}{\mathfrak{N}(\mathfrak{p}^r)}\Big)^{n-j} \frac{1}{\mathfrak{N}(\mathfrak{p}^r)^j}\Big],$$

\noindent and applying Lemma \ref{thm:binomial} with $x = \frac{1}{\mathfrak{N}(\mathfrak{p}^r)}$.

In order to estimate the second series, fix $I$ and assume without loss of generality that $I = \{1, 2, \dots, t\}$. Then, it follows that $\mathfrak{N}(\mathfrak{d}_1), \dots, \mathfrak{N}(\mathfrak{d}_t) > x$, and $\mathfrak{N}(\mathfrak{d}_{t+1}), \dots, \mathfrak{N}(\mathfrak{d}_n) \leq x$ where $t \geq 1$. We estimate the following sum in cases:
$$B := \sum_{\substack{\mathfrak{N}(\mathfrak{d}_i) > x, \, i = 1, \dots, t \\ \mathfrak{N}(\mathfrak{d}_j) \leq x, \, j = t+1, \dots, n}} \frac{|\psi_{n,k,r}(\mathfrak{d}_1, \dots, \mathfrak{d}_n)|}{\mathfrak{N}(\mathfrak{d}_1) \dots \mathfrak{N}(\mathfrak{d}_n)}.$$

First, assume that $k \geq 3$. Then, we obtain via Lemma \ref{thm:convergence}
$$B < \frac{1}{x} \sum_{\mathfrak{d}_1, \dots, \mathfrak{d}_n} \frac{|\psi_{n,k,r}(\mathfrak{d}_1, \dots, \mathfrak{d}_n)|}{\mathfrak{N}(\mathfrak{d}_2) \dots \mathfrak{N}(\mathfrak{d}_n)} = O\Big(\frac{1}{x}\Big).$$

Next, suppose that $k = 2$ and $t \geq 3$. By fixing $0 < \delta < \frac{1}{2}$, Lemma \ref{thm:convergence} yields
\begin{align*}
B &= \sum_{\substack{\mathfrak{N}(\mathfrak{d}_i) > x, \, i = 1, \dots, t \\ \mathfrak{N}(\mathfrak{d}_j) \leq x, \, j = t+1, \dots, n}} \frac{|\psi_{n,2,r}(\mathfrak{d}_1, \dots, \mathfrak{d}_n)| \; \mathfrak{N}(\mathfrak{d}_1)^{\delta - \frac{1}{2}} \dots \mathfrak{N}(\mathfrak{d}_t)^{\delta - \frac{1}{2}}}{\mathfrak{N}(\mathfrak{d}_1)^{\delta + \frac{1}{2}} \dots \mathfrak{N}(\mathfrak{d}_t)^{\delta + \frac{1}{2}} \mathfrak{N}(\mathfrak{d}_{t+1}) \dots \mathfrak{N}(\mathfrak{d}_n)}\\
& < x^{t(\delta - \frac{1}{2})} \sum_{\mathfrak{d}_1, \dots, \mathfrak{d}_n}  \frac{|\psi_{n,2,r}(\mathfrak{d}_1, \dots, \mathfrak{d}_n)|} {\mathfrak{N}(\mathfrak{d}_1)^{\delta + \frac{1}{2}} \dots \mathfrak{N}(\mathfrak{d}_t)^{\delta + \frac{1}{2}} \mathfrak{N}(\mathfrak{d}_{t+1}) \dots \mathfrak{N}(\mathfrak{d}_n)}\\
&= O(x^{t(\delta - \frac{1}{2})}).
\end{align*}

\noindent Since $t \geq 3$ and $0 < \delta < \frac{1}{2}$, we see that $t(\delta - \frac{1}{2}) < -1$, and we conclude that $B = O(\frac{1}{x})$.

Now, suppose that $k = 2$ and $t = 1$. Without loss of generality, let $\mathfrak{N}(\mathfrak{d}_1) > x$ and $\mathfrak{N}(\mathfrak{d}_2), \dots, \mathfrak{N}(\mathfrak{d}_n) \leq x$, and fix a prime ideal $\mathfrak{p}$.
If $\mathfrak{p} \mid \mathfrak{d}_i$ for some $i = 2, \dots, n$, then $\mathfrak{N}(\mathfrak{p}) \leq x$. If $\mathfrak{p} \mid \mathfrak{d}_1$ and $\mathfrak{N}(\mathfrak{p}) > x$, then $\mathfrak{p} \nmid \mathfrak{d}_i$ for all $i = 2, \dots, n$, and thus $\psi_{n,2,r}(\mathfrak{d}_1, \dots, \mathfrak{d}_n) = 0$. Hence, it suffices to consider the prime ideals with norm at most $x$. This implies that

\begin{align*}
B &< \frac{1}{x} \sum_{\substack{\mathfrak{N}(\mathfrak{d}_1) > x \\ \mathfrak{N}(\mathfrak{d}_j) \leq x, \, j = 2, \dots, n}} \frac{|\psi_{n,2,r}(\mathfrak{d}_1, \dots, \mathfrak{d}_n)|} {\mathfrak{N}(\mathfrak{d}_2) \dots \mathfrak{N}(\mathfrak{d}_n)}\\
& \leq \frac{1}{x} \prod_{\mathfrak{N}(\mathfrak{p}) \leq x} \sum_{\nu_1, \dots, \nu_n = 0}^{\infty}
\frac{|\psi_{n,2,r}(\mathfrak{p}^{\nu_1}, \dots, \mathfrak{p}^{\nu_n})|} {\mathfrak{N}(\mathfrak{p})^{\nu_2 + \dots + \nu_n}}.
\end{align*}

\noindent Noting that this resulting series is reminiscent of estimating $Q_{n,2,r}(x)$, we obtain
$$B = \begin{cases} O\Big(\frac{\log^{n-1}{x}}{x}\Big) & \text{if } r = 1,\\ O\Big(\frac{1}{x}\Big) & \text{if } r \geq 2.\end{cases}$$

Finally, suppose that $k = 2$ and $t = 2$. We examine the sum $B$ in two parts: $B_1$ which has $\mathfrak{N}(\mathfrak{d}_1) > x^{3/2}$, and $B_2$ which has $\mathfrak{N}(\mathfrak{d}_1) \leq x^{3/2}$. Estimating $B_1$ is fairly straightforward:
\begin{align*}
B_1
&= \sum_{\substack{\mathfrak{N}(\mathfrak{d}_1) > x^{3/2}, \, \mathfrak{N}(\mathfrak{d}_2) > x \\ \mathfrak{N}(\mathfrak{d}_j) \leq x, \, j = 3, \dots, n}}
\frac{1}{\mathfrak{N}(\mathfrak{d}_1)^{2/3}} \cdot
\frac{|\psi_{n,2,r}(\mathfrak{d}_1, \dots, \mathfrak{d}_n)|} {\mathfrak{N}(\mathfrak{d}_1)^{1/3} \mathfrak{N}(\mathfrak{d}_2) \dots \mathfrak{N}(\mathfrak{d}_n)}\\
&< \frac{1}{x} \sum_{\mathfrak{d}_1, \dots, \mathfrak{d}_n} \frac{|\psi_{n,2,r}(\mathfrak{d}_1, \dots, \mathfrak{d}_n)|} {\mathfrak{N}(\mathfrak{d}_1)^{1/3} \mathfrak{N}(\mathfrak{d}_2) \dots \mathfrak{N}(\mathfrak{d}_n)}\\
&= O\Big(\frac{1}{x}\Big).
\end{align*}

\noindent As for $B_2$, we immediately have
$$B_2 < \frac{1}{x} \sum_{\substack{\mathfrak{N}(\mathfrak{d}_1) \leq x^{3/2}, \, \mathfrak{N}(\mathfrak{d}_2) > x \\ \mathfrak{N}(\mathfrak{d}_j) \leq x, \, j = 3, \dots, n}} \frac{|\psi_{n,2,r}(\mathfrak{d}_1, \dots, \mathfrak{d}_n)|} {\mathfrak{N}(\mathfrak{d}_1) \mathfrak{N}(\mathfrak{d}_3) \dots \mathfrak{N}(\mathfrak{d}_n)}.$$

\noindent Consider a prime ideal $\mathfrak{p}$.
If $\mathfrak{p} \mid \mathfrak{d}_i$ for some $i = 1, 3, \dots, n$, then $\mathfrak{N}(\mathfrak{p}) \leq x^{3/2}$. If $\mathfrak{p} \mid \mathfrak{d}_2$ and $\mathfrak{N}(\mathfrak{p}) > x^{3/2}$, then $\mathfrak{p} \nmid \mathfrak{d}_i$ for all $i = 1, 3, \dots, n$, and thus $\psi_{n,2,r}(\mathfrak{d}_1, \dots, \mathfrak{d}_n) = 0$. Hence, it suffices to consider the prime ideals with norm at most $x^{3/2}$. We obtain
$$B_2 < \frac{1}{x} \prod_{\mathfrak{N}(\mathfrak{p}) \leq x^{3/2}} \sum_{\nu_1, \dots, \nu_n = 0}^{\infty}
\frac{|\psi_{n,2,r}(\mathfrak{p}^{\nu_1}, \dots, \mathfrak{p}^{\nu_n})|} {\mathfrak{N}(\mathfrak{p})^{\nu_1 + \dots + \nu_n}}.$$

\noindent As we have seen previously (as in estimating $Q_{n,2,r}(x)$), this yields
$$B_2 = \begin{cases} \frac{1}{x} O(\log^{n-1}(x^{3/2})) = O\Big(\frac{\log^{n-1}{x}}{x}\Big) & \text{if } r = 1,\\ \frac{1}{x} O(1) = O\Big(\frac{1}{x}\Big) & \text{if } r \geq 2.\end{cases}$$

In summary,
$$\sum_{\mathfrak{N}(\mathfrak{d}_1), \dots, \mathfrak{N}(\mathfrak{d}_n) \leq x}
\frac{\psi_{n,k,r}(\mathfrak{d}_1, \dots, \mathfrak{d}_n)}{\mathfrak{N}(\mathfrak{d}_1) \dots \mathfrak{N}(\mathfrak{d}_n)} = P_{n,k,r} + \begin{cases}
O(\frac{\log^{n-1}{x}}{x}) & \text{if }k = 2 \text{ and } r = 1,\\ O(\frac{1}{x})& \text{otherwise}.\end{cases}$$

\noindent Consequently, we conclude that
$$\sum_{\mathfrak{N}(\mathfrak{a}_1), \dots, \mathfrak{N}(\mathfrak{a}_n) \leq x} \rho_{n,k,r}(\mathfrak{a}_1, \dots, \mathfrak{a}_n) = P_{n,k,r} \cdot (cx)^n + R_{n,k,r}(x),$$

\noindent where $$R_{n,k,r}(x) =
\begin{cases}
O(x^{n-\epsilon} \log^{n-1}{x}) & \text{if }k = 2 \text{ and } r = 1, \\ O(x^{n-\epsilon})& \text{otherwise}.\end{cases}$$

\end{proof}

\section{Approximating the Probabilities}

Inspired by \cite{Toth1}, we now give a way to approximate the infinite products developed in this article. Observe by the Binomial Theorem that for all $k \leq n$, we have
$$\sum_{j=0}^{k-1} \binom{n}{j} \Big(1 - \frac{1}{\mathfrak{N}(\mathfrak{p}^r)}\Big)^{n-j} \frac{1}{\mathfrak{N}(\mathfrak{p}^r)^j} < 1.$$

\noindent Hence, we can take negative logarithms (to create a series with positive terms):
$$-\log{P_{n,k,r}} = \sum_{\mathfrak{p}} \log\Big(\Big[\sum_{j=0}^{k-1} \binom{n}{j} \Big(1 - \frac{1}{\mathfrak{N}(\mathfrak{p}^r)}\Big)^{n-j} \frac{1}{\mathfrak{N}(\mathfrak{p}^r)^j} \Big]^{-1}\Big).$$

\noindent However, for all $k \geq 2$:
\begin{align*}
&\sum_{j=0}^{k-1} \binom{n}{j} \Big(1 - \frac{1}{\mathfrak{N}(\mathfrak{p}^r)}\Big)^{n-j} \frac{1}{\mathfrak{N}(\mathfrak{p}^r)^j}\\
& \geq \Big(1 - \frac{1}{\mathfrak{N}(\mathfrak{p}^r)}\Big)^n + n\Big(1 - \frac{1}{\mathfrak{N}(\mathfrak{p}^r)}\Big)^{n-1} \frac{1}{\mathfrak{N}(\mathfrak{p}^r)}\\
& = \Big(1 - \frac{1}{\mathfrak{N}(\mathfrak{p}^r)}\Big)^{n-1} \Big(1 + \frac{n-1}{\mathfrak{N}(\mathfrak{p}^r)}\Big)\\
& \geq \Big(1 - \frac{n-1}{\mathfrak{N}(\mathfrak{p}^r)}\Big) \Big(1 + \frac{n-1}{\mathfrak{N}(\mathfrak{p}^r)}\Big)
\text{ by Bernoulli's inequality}\\
&= 1 - \Big(\frac{n-1}{\mathfrak{N}(\mathfrak{p}^r)}\Big)^2.
\end{align*}

\noindent Now, let $p_N$ denote that $N$-th rational prime, and let $R_N$ denote the error of $-\log{P_{n,k,r}}$ from truncating the series to summing over all prime ideals with norm at most equal to $p_N$. Then, by taking $N > n-1$ and thus $p_N > n-1$, we obtain
\begin{align*}
R_N &= \sum_{\mathfrak{N}(\mathfrak{p}) \geq p_{N+1}} \log\Big(\Big[\sum_{j=0}^{k-1} \binom{n}{j} \Big(1 - \frac{1}{\mathfrak{N}(\mathfrak{p}^r)}\Big)^{n-j} \frac{1}{\mathfrak{N}(\mathfrak{p}^r)^j} \Big]^{-1}\Big) \\
&\leq \sum_{\mathfrak{N}(\mathfrak{p}) \geq p_{N+1}} \log\Big(\Big[1 - \Big(\frac{n-1} {\mathfrak{N}(\mathfrak{p}^r)} \Big)^2 \Big]^{-1}\Big)\\
&= \sum_{\mathfrak{N}(\mathfrak{p}) \geq p_{N+1}} \log\Big(1 + \frac{(n - 1)^2} {\mathfrak{N}(\mathfrak{p}^r)^2
- (n-1)^2}\Big)\\
&\leq \sum_{\mathfrak{N}(\mathfrak{p}) \geq p_{N+1}} \frac{(n - 1)^2} {\mathfrak{N}(\mathfrak{p}^r)^2
- (n-1)^2}\\
&\leq d \sum_{j=N+1}^{\infty} \frac{(n-1)^2}{p_j^2 - (n-1)^2}.
\end{align*}

\noindent For the last line, we use the fact that there are at most $d = [K: \mathbb{Q}]$ prime ideals lying above a given rational prime. Next, since $p_j > 2j$ for all $j \geq 5$, we obtain
$$R_N \leq d \sum_{j=N+1}^{\infty} \frac{(n-1)^2}{(2j)^2 - (n-1)^2}.$$

\noindent Next, we rewrite this as a telescoping sum, obtaining
\begin{align*}
R_N &\leq \frac{d(n-1)}{2} \sum_{j=N+1}^{\infty} \Big(\frac{1}{2j-(n-1)} - \frac{1}{2j+(n-1)}\Big)\\
&= \frac{d(n-1)}{2}  \Big(\frac{1}{2N-n+3} + \frac{1}{2N-n+5} + \dots + \frac{1}{2N+n-1}\Big)\\
&\leq \frac{d(n-1)^2}{2(2N-n+3)}.
\end{align*}

\noindent Finally, to guarantee $t$ decimal point accuracy to $P_{n,k,r}$, we want $\displaystyle R_N \leq \frac{10^{-t}}{2}$. By using the work above, we find that $\displaystyle N \geq \frac{d(n-1)^2 \cdot 10^t + (n-3)}{2}$.

Using Python, we obtain approximate values for probabilities of ordered $n$-tuples of ideals from a few algebraic integer rings being pairwise relatively prime, rounded to the nearest ten-thousandth.

\begin{figure}[!htb]
%\begin{table}
\centering
\begin{tabular}{|l|c|c|c|c|c|}
\hline
 & $\mathbb{Z}$ & $\mathbb{Z}[\sqrt{2}]$ & $\mathbb{Z}[i]$ & $\mathbb{Z}[e^{2\pi i/3}]$ & $\mathbb{Z}[e^{2\pi i/5}]$\\
\hline
$n=2$ & 0.6079 & 0.6969 & 0.6637 & 0.7781  & 0.9155\\
$n=3$ & 0.2867 & 0.4066 & 0.3572 & 0.5151 & 0.7818\\
$n=4$ & 0.1149 & 0.2115 & 0.1676 & 0.3035 & 0.6307\\
\hline
\end{tabular}
\captionof{figure}{Sample pairwise relatively prime probabilities.}
%\end{table}
\end{figure}


\begin{thebibliography}{1}

\bibitem{Benkoski} S.J. Benkoski, The probability that $k$ positive integers are relatively $r$-prime, \emph{J. Number Theory} \textbf{8} (1976) 218-223.

%\bibitem{DF} D. Dummit and R. Foote, Abstract Algebra, 3rd ed., John Wiley \& Sons, 2003.

\bibitem{Hu} J. Hu, The probability that random positive integers are $k$-wise relatively prime, \emph{Int. J. Number Theory} \textbf{9} (2013) 1263-1271.

\bibitem{Landau} E. Landau, Einf\"{u}hrung in die Elementare und Theorie der Algebraischen Zahlen und der Ideale, 2. Aufl. Leipzig.

%\bibitem{Marcus} D.A. Marcus, Number Fields, 3nd ed., Springer-Verlag, 1995.

%\bibitem{Murty} M.R. Murty, Problems in Analytic Number Theory, 2nd ed., Springer, 2007.

\bibitem{Nymann} J.E. Nymann, On the probability that $k$ positive integers are relatively prime,
\emph{J. Number Theory} \textbf{4} (1972) 469-473.

\bibitem{Rosen} M. Rosen, A Generalization of Mertens' Theorem, \emph{J. Ramanujan Math. Soc.} \textbf{14} (1999) 1-19.

\bibitem{Sittinger} B. Sittinger, The probability that random algebraic integers are relatively $r$-prime, \emph{J. Number Theory} \textbf{130} (2010) 164-171.

%\bibitem{Stewart} I. Stewart and D. Tall, Algebraic Number Theory and Fermat's Last Theorem, 3rd ed., AK Peters/CRC Press, 2001.

\bibitem{Toth1} L. T\'oth, The probability that $k$ positive integers are pairwise relatively prime, \emph{Fibonacci Quart.} \textbf{40} (2002) 13-18.

\bibitem{Toth2} L. T\'oth, Counting $r$-tuples of positive integers with $k$-wise relatively prime components, \emph{J. Number Theory} \textbf{166} (2016) 105-116.

\end{thebibliography}
\end{document}